\documentclass{article}
\usepackage{amsmath,amsthm,amsopn,amstext,amscd,amsfonts,amssymb,verbatim}
\usepackage{pdflscape}
\usepackage{natbib}

\def \u {\mathop{\rm \mathcal{U}}\nolimits}
\def \tr {\mathop{\rm tr}\nolimits}
\def \re {\mathop{\rm Re}\nolimits}

\def \Vol {\mathop{\rm Vol}\nolimits}

\def \etr {\mathop{\rm etr}\nolimits}
\def \diag {\mathop{\rm diag}\nolimits}

\renewenvironment{abstract}
                 {\vspace{6pt}
                  \begin{center}
                  \begin{minipage}{5in}
                  \centerline{\textbf{Abstract}}
                  \noindent\ignorespaces
                 }
                 {\end{minipage}\end{center}}
\newtheorem{theorem}{\textbf{Theorem}}[section]
\newtheorem{corollary}{\textbf{Corollary}}[section]

\newtheorem{proposition}{\textbf{Proposition}}[section]
\theoremstyle{definition}
\newtheorem{definition}{\textbf{Definition}}[section]

\newtheorem{remark}{\textbf{Remark}}[section]

\title{\Large \textbf{Generalised matricvariate $T$-distribution}}
\author{
  \textbf{Jos\'e A. D\'{\i}az-Garc\'{\i}a} \thanks{Corresponding author\newline
    }\\
  {\normalsize Department of Statistics and Computation} \\
  {\normalsize Universidad Autonoma Agraria Antonio Narro} \\
  {\normalsize 25350 Buenavista, Saltillo, Coahuila, Mexico} \\
  {\normalsize E-mail: jadiaz@uaaan.mx} \\[2ex]
  \textbf{Ram\'on Guti\'errez-S\'anchez} \\
  {\normalsize Department of Statistics and O.R} \\
  {\normalsize University of Granada} \\
  {\normalsize Granada 18071, Spain}\\
  {\normalsize E-mail: ramongs@ugr.es}\\
}
\date{}
\begin{document}
\maketitle

\begin{abstract}
Assuming Kotz-Riesz type I and II distributions and their corresponding independent Riesz
distributions the associated generalised matricvariate  $T$ distributions, termed matricvariate
$T$-Riesz distributions for real normed division algebras are obtained with respect to the
Lebesgue measure. In addition some of their properties are also studied.
\end{abstract}

\section{Introduction}\label{sec1}

Since the early 80's years the elliptical distribution family has been used as an alternative sampling
distribution for the classical restriction of normality. The elliptical distribution family, within other
quantities  of interest, are very attractive because if it is assumed that the random matrix, say
$\mathbf{X}$, has a matrix multivariate elliptical distribution, then  distributions of several matrix
functions, $\mathbf{Y} = f(\mathbf{X})$, are invariant under the family of elliptical distributions;
furthermore, such distributions coincide with those obtained when $\mathbf{X}$ is distributed according
to a matrix multivariate normal distribution, see \citet{fz:90} and \citet{gv:93}.

However it should be noted that the invariance described above occurs when a probabilistically dependency
is assumed. This is, if the matrix $\mathbf{X} = \left[
\begin{array}{c}
  \mathbf{X}_{1} \\
  \mathbf{X}_{2}
\end{array}
\right]$ has a matrix multivariate elliptical distribution, observing that $\mathbf{X}_{1}$ and
$\mathbf{X}_{2}$ are probabilistically dependent (note that $\mathbf{X}_{1}$ and $\mathbf{X}_{2}$ are
probabilistically independent if $\mathbf{X}$ has a matrix multivariate normal distribution,
\citet{gv:93}). Then if $\mathbf{X}'$ denotes the transpose of $\mathbf{X}$ and if $\mathbf{A}$ is
non-negative definite matrix and $\u(\mathbf{A})$ denotes a upper triangular matrix, such that
$\mathbf{A}=\u(\mathbf{A})^{'}\u(\mathbf{A})$ defines the Cholesky's decomposition of $\mathbf{A}$
\citep{m:82}, then:
\begin{itemize}
  \item Let $\mathbf{T} = \mathbf{X}_{1}\u(\mathbf{X}'_{2}\mathbf{X}_{2})^{-1}$. It is said
  that $\mathbf{T}$ is distributed according to a matricvariate \footnote{The term matricvariate
  distribution was first introduced \citet{di:67}, but the expression matrix-variate distribution
  or matrix variate distribution or matrix multivariate distribution was later used to describe
  any distribution of a random matrix, see \citet{gn:00}, and references therein. When the density
  function of a random matrix is written only in terms of determinant operator and $q_{\kappa}(\cdot)$
  (defined in the next section) then the matricvariate designation shall be used} $T$- distribution.

  \item Let $\mathbf{F} =  \u(\mathbf{X}'_{2}\mathbf{X}_{2})^{'-1}(\mathbf{X}'_{1}\mathbf{X}_{1})
  \u(\mathbf{X}'_{2}\mathbf{X}_{2})^{-1}$. It is said that $\mathbf{F}$ has a matricvariate
  beta type II distribution, and

  \item let $\mathbf{B} = \u\left(\mathbf{X}'_{1}\mathbf{X}_{1} + \mathbf{X}'_{2}\mathbf{X}_{2}
  \right)^{'-1}(\mathbf{X}'_{1}\mathbf{X}_{1})\u(\mathbf{X}'_{1}\mathbf{X}_{1} +
  \mathbf{X}'_{2}\mathbf{X}_{2})^{-1}$. It is said
  that $\mathbf{B}$ is distributed according to a matricvariate  beta type I distribution,
\end{itemize}
where the matricvariate $T$, beta type II and beta type I distributions are the same distributions as
those obtained when $\mathbf{X}$ has a matrix multivariate normal distribution, see \citet{fz:90} and
\citet{gv:93}.

Unfortunately, in some situations, $\mathbf{X}_{1}$ and $\mathbf{X}_{2}$ are assumed independent. This
situation can occur in the context of multivariate Bayesian inference, \citet{p:82}. For example, suppose
that a particular distribution depend of two matrix parameters, say $\boldsymbol{\theta_{1}}$ and
$\boldsymbol{\theta_{2}}$; and is assumed that $\boldsymbol{\theta_{1}}$ and $\boldsymbol{\theta_{2}}$
have a distribution as the marginal distributions of $\mathbf{X}_{1}$ and $\mathbf{X}_{2}$ respectively,
but in this case is assumed that $\boldsymbol{\theta_{1}}$ and $\boldsymbol{\theta_{2}}$ are
\emph{independent}. Under this hypothesis, one is interested in finding the priori distribution of a
parameter type $\mathbf{T}$, $\mathbf{F}$ or $\mathbf{B}$ in terms of the priori distribution of
$\boldsymbol{\theta_{1}}$ and $\boldsymbol{\theta_{2}}$. In this case priori distributions of
$\mathbf{T}$, $\mathbf{F}$ or $\mathbf{B}$ are different from those obtained under dependence, moreover,
such priori distributions are different under each particular elliptical distribution assumed.

Now, in the matrix multivariate elliptical distribution setting the so termed Kotz-Riesz distribution
involves some importance by its relation with the Riesz distribution. If $\mathbf{X}$ is distributed
according to matrix multivariate Kotz-Riesz, then the matrix $\mathbf{V} = \mathbf{X}^{'}\mathbf{X}$ has
a Riesz distribution, \citet{dg:15c}. The Riesz distributions, were  first introduced by \citet{hl:01}
under the name of Riesz natural exponential family (Riesz NEF); it was based on a special case of the
so-called Riesz measure from \citet[p.137]{fk:94}. This Riesz distribution generalises the matrix
multivariate gamma and Wishart distributions, containing them as particular cases.

In analogy with the case of $T$-distribution obtained under normality, there exist two possible
generalisations of it when a Kotz-Riesz distribution is assumed, the matricvariate $T$-distribution and
the matrix multivariate $T$-distribution, see \citet{dggj:12}. In particular this paper will focus on the
matricvariate $T$-Riesz distribution.

The paper is organized as follows; first,  some basic concepts and the notation of abstract algebra and
Jacobians are summarised in Section \ref{sec2}. Then the nonsingular central matricvariate $T$-Riesz and
the corresponding generalised beta type II distributions and some of their basic properties are studied
in sections \ref{sec3} and \ref{sec4}, respectively. It should be noted that all these results are
derived in the context of  real normed division algebras, a useful integrated and unified approach
recently implemented in matrix distribution theory.

\section{Preliminary results}\label{sec2}

A detailed discussion of real normed division algebras may be found in \citet{b:02} and \citet{E:90}. For
convenience, we shall introduce some notation, although in general we adhere to standard notation forms.

For our purposes: Let $\mathbb{F}$ be a field. An \emph{algebra} $\mathfrak{A}$ over $\mathbb{F}$ is a
pair $(\mathfrak{A};m)$, where $\mathfrak{A}$ is a \emph{finite-dimensional vector space} over
$\mathbb{F}$ and \emph{multiplication} $m : \mathfrak{A} \times \mathfrak{A} \rightarrow A$ is an
$\mathbb{F}$-bilinear map; that is, for all $\lambda \in \mathbb{F},$ $x, y, z \in \mathfrak{A}$,
\begin{eqnarray*}
  m(x, \lambda y + z) &=& \lambda m(x; y) + m(x; z) \\
  m(\lambda x + y; z) &=& \lambda m(x; z) + m(y; z).
\end{eqnarray*}
Two algebras $(\mathfrak{A};m)$ and $(\mathfrak{E}; n)$ over $\mathbb{F}$ are said to be
\emph{isomorphic} if there is an invertible map $\phi: \mathfrak{A} \rightarrow \mathfrak{E}$ such that
for all $x, y \in \mathfrak{A}$,
$$
  \phi(m(x, y)) = n(\phi(x), \phi(y)).
$$
By simplicity, we write $m(x; y) = xy$ for all $x, y \in \mathfrak{A}$.

Let $\mathfrak{A}$ be an algebra over $\mathbb{F}$. Then $\mathfrak{A}$ is said to be
\begin{enumerate}
  \item \emph{alternative} if $x(xy) = (xx)y$ and $x(yy) = (xy)y$ for all $x, y \in \mathfrak{A}$,
  \item \emph{associative} if $x(yz) = (xy)z$ for all $x, y, z \in \mathfrak{A}$,
  \item \emph{commutative} if $xy = yx$ for all $x, y \in \mathfrak{A}$, and
  \item \emph{unital} if there is a $1 \in \mathfrak{A}$ such that $x1 = x = 1x$ for all $x \in \mathfrak{A}$.
\end{enumerate}
If $\mathfrak{A}$ is unital, then the identity 1 is uniquely determined.

An algebra $\mathfrak{A}$ over $\mathbb{F}$ is said to be a \emph{division algebra} if $\mathfrak{A}$ is
nonzero and $xy = 0_{\mathfrak{A}} \Rightarrow x = 0_{\mathfrak{A}}$ or $y = 0_{\mathfrak{A}}$ for all
$x, y \in \mathfrak{A}$.

The term ``division algebra", comes from the following proposition, which shows that, in such an algebra,
left and right division can be unambiguously performed.

Let $\mathfrak{A}$ be an algebra over $\mathbb{F}$. Then $\mathfrak{A}$ is a division algebra if, and
only if, $\mathfrak{A}$ is nonzero and for all $a, b \in \mathfrak{A}$, with $b \neq 0_{\mathfrak{A}}$,
the equations $bx = a$ and $yb = a$ have unique solutions $x, y \in \mathfrak{A}$.

In the sequel we assume $\mathbb{F} = \Re$ and consider classes of division algebras over $\Re$ or
``\emph{real division algebras}" for short.

We introduce the algebras of \emph{real numbers} $\Re$, \emph{complex numbers} $\mathfrak{C}$,
\emph{quaternions} $\mathfrak{H}$ and \emph{octonions} $\mathfrak{O}$. Then, if $\mathfrak{A}$ is an
alternative real division algebra, then $\mathfrak{A}$ is isomorphic to $\Re$, $\mathfrak{C}$,
$\mathfrak{H}$ or $\mathfrak{O}$.

Let $\mathfrak{A}$ be a real division algebra with identity $1$. Then $\mathfrak{A}$ is said to be
\emph{normed} if there is an inner product $(\cdot, \cdot)$ on $\mathfrak{A}$ such that
$$
  (xy, xy) = (x, x)(y, y) \qquad \mbox{for all } x, y \in \mathfrak{A}.
$$
If $\mathfrak{A}$ is a \emph{real normed division algebra}, then $\mathfrak{A}$ is isomorphic $\Re$,
$\mathfrak{C}$, $\mathfrak{H}$ or $\mathfrak{O}$.

There are exactly four normed division algebras: real numbers ($\Re$), complex numbers ($\mathfrak{C}$),
quaternions ($\mathfrak{H}$) and octonions ($\mathfrak{O}$), see \citet{b:02}. We take into account that
should be taken into account, $\Re$, $\mathfrak{C}$, $\mathfrak{H}$ and $\mathfrak{O}$ are the only
normed division algebras; furthermore, they are the only alternative division algebras.

Let $\mathfrak{A}$ be a division algebra over the real numbers. Then $\mathfrak{A}$ has dimension either
1, 2, 4 or 8. In other branches of mathematics, the parameters $\alpha = 2/\beta$ and $t = \beta/4$ are
used, see \citet{er:05} and \citet{k:84}, respectively.

Finally, observe that

\begin{tabular}{c}
  $\Re$ is a real commutative associative normed division algebras, \\
  $\mathfrak{C}$ is a commutative associative normed division algebras,\\
  $\mathfrak{H}$ is an associative normed division algebras, \\
  $\mathfrak{O}$ is an alternative normed division algebras. \\
\end{tabular}

Let $\mathfrak{L}^{\beta}_{n,m}$ be the set of all $n \times m$ matrices of rank $m \leq n$ over
$\mathfrak{A}$ with $m$ distinct positive singular values, where $\mathfrak{A}$ denotes a \emph{real
finite-dimensional normed division algebra}. Let $\mathfrak{A}^{n \times m}$ be the set of all $n \times
m$ matrices over $\mathfrak{A}$. The dimension of $\mathfrak{A}^{n \times m}$ over $\Re$ is $\beta mn$.
Let $\mathbf{A} \in \mathfrak{A}^{n \times m}$, then $\mathbf{A}^{*} = \bar{\mathbf{A}}^{T}$ denotes the
usual conjugate transpose.

Table \ref{table1} sets out the equivalence between the same concepts in the four normed division
algebras.

\begin{table}[th]
  \centering
  \caption{\scriptsize Notation}\label{table1}
  \begin{scriptsize}
  \begin{tabular}{cccc|c}
    \hline
    Real & Complex & Quaternion & Octonion & \begin{tabular}{c}
                                               Generic \\
                                               notation \\
                                             \end{tabular}\\
    \hline
    Semi-orthogonal & Semi-unitary & Semi-symplectic & \begin{tabular}{c}
                                                         Semi-exceptional \\
                                                         type \\
                                                       \end{tabular}
      & $\mathcal{V}_{m,n}^{\beta}$ \\
    Orthogonal & Unitary & Symplectic & \begin{tabular}{c}
                                                         Exceptional \\
                                                         type \\
                                                       \end{tabular} & $\mathfrak{U}^{\beta}(m)$ \\
    Symmetric & Hermitian & \begin{tabular}{c}
                              Quaternion \\
                              hermitian \\
                            \end{tabular}
     & \begin{tabular}{c}
                              Octonion \\
                              hermitian \\
                            \end{tabular} & $\mathfrak{S}_{m}^{\beta}$ \\
    \hline
  \end{tabular}
  \end{scriptsize}
\end{table}

We denote by ${\mathfrak S}_{m}^{\beta}$ the real vector space of all $\mathbf{S} \in \mathfrak{A}^{m
\times m}$ such that $\mathbf{S} = \mathbf{S}^{*}$. In addition, let $\mathfrak{P}_{m}^{\beta}$ be the
\emph{cone of positive definite matrices} $\mathbf{S} \in \mathfrak{A}^{m \times m}$. Thus,
$\mathfrak{P}_{m}^{\beta}$ consist of all matrices $\mathbf{S} = \mathbf{X}^{*}\mathbf{X}$, with
$\mathbf{X} \in \mathfrak{L}^{\beta}_{n,m}$; then $\mathfrak{P}_{m}^{\beta}$ is an open subset of
${\mathfrak S}_{m}^{\beta}$.

Let $\mathfrak{D}_{m}^{\beta}$ consisting of all $\mathbf{D} \in \mathfrak{A}^{m \times m}$, $\mathbf{D}
= \diag(d_{1}, \dots,d_{m})$. Let $\mathfrak{T}_{U}^{\beta}(m)$ be the subgroup of all \emph{upper
triangular} matrices $\mathbf{T} \in \mathfrak{A}^{m \times m}$ such that $t_{ij} = 0$ for $1 < i < j
\leq m$.

For any matrix $\mathbf{X} \in \mathfrak{A}^{n \times m}$, $d\mathbf{X}$ denotes the\emph{ matrix of
differentials} $(dx_{ij})$. Finally, we define the \emph{measure} or volume element $(d\mathbf{X})$ when
$\mathbf{X} \in \mathfrak{A}^{n \times m}, \mathfrak{S}_{m}^{\beta}$, $\mathfrak{D}_{m}^{\beta}$ or
$\mathcal{V}_{m,n}^{\beta}$, see \citet{dggj:11} and \citet{dggj:13}.

If $\mathbf{X} \in \mathfrak{A}^{n \times m}$ then $(d\mathbf{X})$ (the Lebesgue measure in
$\mathfrak{A}^{n \times m}$) denotes the exterior product of the $\beta mn$ functionally independent
variables
$$
  (d\mathbf{X}) = \bigwedge_{i = 1}^{n}\bigwedge_{j = 1}^{m}dx_{ij} \quad \mbox{ where }
    \quad dx_{ij} = \bigwedge_{k = 1}^{\beta}dx_{ij}^{(k)}.
$$

If $\mathbf{S} \in \mathfrak{S}_{m}^{\beta}$ (or $\mathbf{S} \in \mathfrak{T}_{U}^{\beta}(m)$ with
$t_{ii} >0$, $i = 1, \dots,m$) then $(d\mathbf{S})$ (the Lebesgue measure in $\mathfrak{S}_{m}^{\beta}$
or in $\mathfrak{T}_{U}^{\beta}(m)$) denotes the exterior product of the exterior product of the
$m(m-1)\beta/2 + m$ functionally independent variables,
$$
  (d\mathbf{S}) = \bigwedge_{i=1}^{m} ds_{ii}\bigwedge_{i > j}^{m}\bigwedge_{k = 1}^{\beta}
                      ds_{ij}^{(k)}.
$$
Observe, that for the Lebesgue measure $(d\mathbf{S})$ defined thus, it is required that $\mathbf{S} \in
\mathfrak{P}_{m}^{\beta}$, that is, $\mathbf{S}$ must be a non singular Hermitian matrix (Hermitian
definite positive matrix).

If $\mathbf{\Lambda} \in \mathfrak{D}_{m}^{\beta}$ then $(d\mathbf{\Lambda})$ (the Legesgue measure in
$\mathfrak{D}_{m}^{\beta}$) denotes the exterior product of the $\beta m$ functionally independent
variables
$$
  (d\mathbf{\Lambda}) = \bigwedge_{i = 1}^{n}\bigwedge_{k = 1}^{\beta}d\lambda_{i}^{(k)}.
$$
If $\mathbf{H}_{1} \in \mathcal{V}_{m,n}^{\beta}$ then
$$
  (\mathbf{H}^{*}_{1}d\mathbf{H}_{1}) = \bigwedge_{i=1}^{m} \bigwedge_{j =i+1}^{n}
  \mathbf{h}_{j}^{*}d\mathbf{h}_{i}.
$$
where $\mathbf{H} = (\mathbf{H}^{*}_{1}|\mathbf{H}^{*}_{2})^{*} = (\mathbf{h}_{1}, \dots,
\mathbf{h}_{m}|\mathbf{h}_{m+1}, \dots, \mathbf{h}_{n})^{*} \in \mathfrak{U}^{\beta}(n)$. It can be
proved that this differential form does not depend on the choice of the $\mathbf{H}_{2}$ matrix. When $n
= 1$; $\mathcal{V}^{\beta}_{m,1}$ defines the unit sphere in $\mathfrak{A}^{m}$. This is, of course, an
$(m-1)\beta$- dimensional surface in $\mathfrak{A}^{m}$. When $n = m$ and denoting $\mathbf{H}_{1}$ by
$\mathbf{H}$, $(\mathbf{H}d\mathbf{H}^{*})$ is termed the \emph{Haar measure} on
$\mathfrak{U}^{\beta}(m)$.

The surface area or volume of the Stiefel manifold $\mathcal{V}^{\beta}_{m,n}$ is
\begin{equation}\label{vol}
    \Vol(\mathcal{V}^{\beta}_{m,n}) = \int_{\mathbf{H}_{1} \in
  \mathcal{V}^{\beta}_{m,n}} (\mathbf{H}_{1}d\mathbf{H}^{*}_{1}) =
  \frac{2^{m}\pi^{mn\beta/2}}{\Gamma^{\beta}_{m}[n\beta/2]},
\end{equation}
where $\Gamma^{\beta}_{m}[a]$ denotes the multivariate \emph{Gamma function} for the space
$\mathfrak{S}_{m}^{\beta}$. This can be obtained as a particular case of the \emph{generalised gamma
function of weight $\kappa$} for the space $\mathfrak{S}^{\beta}_{m}$ with $\kappa = (k_{1}, k_{2},
\dots, k_{m}) \in \Re^{m}$, taking $\kappa =(0,0,\dots,0) \in \Re^{m}$ and which for $\re(a) \geq
(m-1)\beta/2 - k_{m}$ is defined by, see \citet{gr:87} and \citet{fk:94},
\begin{eqnarray}\label{int1}
  \Gamma_{m}^{\beta}[a,\kappa] &=& \displaystyle\int_{\mathbf{A} \in \mathfrak{P}_{m}^{\beta}}
  \etr\{-\mathbf{A}\} |\mathbf{A}|^{a-(m-1)\beta/2 - 1} q_{\kappa}(\mathbf{A}) (d\mathbf{A}) \\
&=& \pi^{m(m-1)\beta/4}\displaystyle\prod_{i=1}^{m} \Gamma[a + k_{i}
    -(i-1)\beta/2]\nonumber\\ \label{gammagen1}
&=& [a]_{\kappa}^{\beta} \Gamma_{m}^{\beta}[a],
\end{eqnarray}
where $\etr(\cdot) = \exp(\tr(\cdot))$, $|\cdot|$ denotes the determinant, and for $\mathbf{A} \in
\mathfrak{S}_{m}^{\beta}$
\begin{equation}\label{hwv}
    q_{\kappa}(\mathbf{A}) = |\mathbf{A}_{m}|^{k_{m}}\prod_{i = 1}^{m-1}|\mathbf{A}_{i}|^{k_{i}-k_{i+1}}
\end{equation}
with $\mathbf{A}_{p} = (a_{rs})$, $r,s = 1, 2, \dots, p$, $p = 1,2, \dots, m$ is termed the \emph{highest
weight vector}, see \citet{gr:87}. Also,
\begin{eqnarray*}
  \Gamma_{m}^{\beta}[a] &=& \displaystyle\int_{\mathbf{A} \in \mathfrak{P}_{m}^{\beta}}
  \etr\{-\mathbf{A}\} |\mathbf{A}|^{a-(m-1)\beta/2 - 1}(d\mathbf{A}) \\ \label{cgamma}
    &=& \pi^{m(m-1)\beta/4}\displaystyle\prod_{i=1}^{m} \Gamma[a-(i-1)\beta/2],
\end{eqnarray*}
and $\re(a)> (m-1)\beta/2$.

In other branches of mathematics the \textit{highest weight vector} $q_{\kappa}(\mathbf{A})$ is also
termed the \emph{generalised power} of $\mathbf{A}$ and is denoted as $\Delta_{\kappa}(\mathbf{A})$, see
\citet{fk:94} and \citet{hl:01}.

Additional properties of $q_{\kappa}(\mathbf{A})$, which are immediate consequences of the definition of
$q_{\kappa}(\mathbf{A})$ are:
\begin{enumerate}
  \item Let $\mathbf{A} = \mathbf{L}^{*}\mathbf{DL}$ be the L'DL decomposition of $\mathbf{A} \in \mathfrak{P}_{m}^{\beta}$,
        where $\mathbf{L} \in \mathfrak{T}_{U}^{\beta}(m)$ with $l_{ii} = 1$, $i = 1, 2, \ldots ,m$ and
        $\mathbf{D} = \diag(\lambda_{1}, \dots, \lambda_{m})$, $\lambda_{i} \geq 0$, $i = 1, 2, \ldots
        ,m$. Then
        \begin{equation}\label{qk1}
          q_{\kappa}(\mathbf{A}) = \prod_{i=1}^{m} \lambda_{i}^{k_{i}}.
        \end{equation}
      \item
      \begin{equation}\label{qk2}
        q_{\kappa}(\mathbf{A}^{-1}) =  q_{-\kappa^{*}}^{*}(\mathbf{A}),
      \end{equation}
      where $\kappa^{*}=(k_{m}, k_{m-1}, \dots,k_{1})$, $-\kappa^{*}=(-k_{m}, -k_{m-1},
      \dots,-k_{1})$,
      \begin{equation}\label{hhwv}
         q_{\kappa}^{*}(\mathbf{A}) = |\mathbf{A}_{m}|^{k_{m}}\prod_{i = 1}^{m-1}|\mathbf{A}_{i}|^{k_{i}-k_{i+1}}
      \end{equation}
      and
      \begin{equation}\label{qqk1}
        q_{\kappa}^{*}(\mathbf{A}) = \prod_{i=1}^{m} \lambda_{i}^{k_{m-i+1}},
      \end{equation}
      see \citet[pp. 126-127 and Proposition VII.1.5]{fk:94}.

  Alternatively, let $\mathbf{A} = \mathbf{T}^{*}\mathbf{T}$ the Cholesky's decomposition of
  matrix $\mathbf{A} \in \mathfrak{P}_{m}^{\beta}$, with $\mathbf{T}=(t_{ij}) \in
  \mathfrak{T}_{U}^{\beta}(m)$, then $\lambda_{i} = t_{ii}^{2}$, $t_{ii} \geq 0$, $i = 1, 2,
  \ldots ,m$. See \citet[p. 931, first paragraph]{hl:01}, \citet[p. 390, lines -11 to
  -16]{hlz:05} and \citet[p.5, lines 1-6]{k:14}.
  \item if $\kappa = (p, \dots, p)$, then
    \begin{equation}\label{qk3}
        q_{\kappa}(\mathbf{A}) = |\mathbf{A}|^{p},
    \end{equation}
    in particular if $p=0$, then $q_{\kappa}(\mathbf{A}) = 1$.
  \item if $\tau = (t_{1}, t_{2}, \dots, t_{m})$, $t_{1}\geq t_{2}\geq \cdots \geq t_{m} \geq
  0$, then
    \begin{equation}\label{qk41}
        q_{\kappa+\tau}(\mathbf{A}) = q_{\kappa}(\mathbf{A})q_{\tau}(\mathbf{A}),
    \end{equation}
    in particular if $\tau = (p,p, \dots, p)$,  then
    \begin{equation}\label{qk42}
        q_{\kappa+\tau}(\mathbf{A}) \equiv q_{\kappa+p}(\mathbf{A}) = |\mathbf{A}|^{p} q_{\kappa}(\mathbf{A}).
    \end{equation}
    \item Finally, for $\mathbf{B} \in \mathfrak{T}_{U}^{\beta}(m)$  in such a manner that $\mathbf{C} =
    \mathbf{B}^{*}\mathbf{B} \in \mathfrak{S}_{m}^{\beta}$,
    \begin{equation}\label{qk5}
        q_{\kappa}(\mathbf{B}^{*}\mathbf{AB}) = q_{\kappa}(\mathbf{C})q_{\kappa}(\mathbf{A})
    \end{equation}
    and
    \begin{equation}\label{qk6}
        q_{\kappa}(\mathbf{B}^{*-1}\mathbf{A}\mathbf{B}^{-1}) = (q_{\kappa}(\mathbf{C}))^{-1}q_{\kappa}(\mathbf{A})
        = q_{-\kappa}(\mathbf{C})q_{\kappa}(\mathbf{A}),
    \end{equation}
see \citet[p. 776, eq. (2.1)]{hlz:08}.
\end{enumerate}
\begin{remark}
Let $\mathcal{P}(\mathfrak{S}_{m}^{\beta})$ denote the algebra of all polynomial functions on
$\mathfrak{S}_{m}^{\beta}$, and $\mathcal{P}_{k}(\mathfrak{S}_{m}^{\beta})$ the subspace of homogeneous
polynomials of degree $k$ and let $\mathcal{P}^{\kappa}(\mathfrak{S}_{m}^{\beta})$ be an irreducible
subspace of $\mathcal{P}(\mathfrak{S}_{m}^{\beta})$ such that
$$
  \mathcal{P}_{k}(\mathfrak{S}_{m}^{\beta}) = \sum_{\kappa}\bigoplus
  \mathcal{P}^{\kappa}(\mathfrak{S}_{m}^{\beta}).
$$
Note that $q_{\kappa}$ is a homogeneous polynomial of degree $k$, moreover $q_{\kappa} \in
\mathcal{P}^{\kappa}(\mathfrak{S}_{m}^{\beta})$, see \citet{gr:87}.
\end{remark}
In (\ref{gammagen1}), $[a]_{\kappa}^{\beta}$ denotes the generalised Pochhammer symbol of weight
$\kappa$, defined as
\begin{eqnarray*}
  [a]_{\kappa}^{\beta} &=& \prod_{i = 1}^{m}(a-(i-1)\beta/2)_{k_{i}}\\
    &=& \frac{\pi^{m(m-1)\beta/4} \displaystyle\prod_{i=1}^{m}
    \Gamma[a + k_{i} -(i-1)\beta/2]}{\Gamma_{m}^{\beta}[a]} \\
    &=& \frac{\Gamma_{m}^{\beta}[a,\kappa]}{\Gamma_{m}^{\beta}[a]},
\end{eqnarray*}
where $\re(a) > (m-1)\beta/2 - k_{m}$ and
$$
  (a)_{i} = a (a+1)\cdots(a+i-1),
$$
is the standard Pochhammer symbol.

An alternative definition of the generalised gamma function of weight $\kappa$ is proposed by
\citet{k:66}, which is defined as%
\begin{eqnarray}\label{int2}
  \Gamma_{m}^{\beta}[a,-\kappa] &=& \displaystyle\int_{\mathbf{A} \in \mathfrak{P}_{m}^{\beta}}
    \etr\{-\mathbf{A}\} |\mathbf{A}|^{a-(m-1)\beta/2 - 1} q_{\kappa}(\mathbf{A}^{-1})
    (d\mathbf{A}) \\
&=& \pi^{m(m-1)\beta/4}\displaystyle\prod_{i=1}^{m} \Gamma[a - k_{i}
    -(m-i)\beta/2] \nonumber\\ \label{gammagen2}
&=& \displaystyle\frac{(-1)^{k} \Gamma_{m}^{\beta}[a]}{[-a +(m-1)\beta/2
    +1]_{\kappa}^{\beta}} ,
\end{eqnarray}
where $\re(a) > (m-1)\beta/2 + k_{1}$.

In addition consider the following generalised beta functions termed, \emph{generalised c-beta function},
see \citet[p. 130]{fk:94} and \citet{dg:15b},
$$
  \mathcal{B}_{m}^{\beta}[a,\kappa;b, \tau] \hspace{10cm}
$$
$$
  \hspace{.5cm}= \int_{\mathbf{0}<\mathbf{S}<\mathbf{I}_{m}}
    |\mathbf{S}|^{a-(m-1)\beta/2-1} q_{\kappa}(\mathbf{S})|\mathbf{I}_{m} - \mathbf{S}|^{b-(m-1)\beta/2-1}
    q_{\tau}(\mathbf{I}_{m} - \mathbf{S})(d\mathbf{S})
$$
$$
   = \int_{\mathbf{R} \in\mathfrak{P}_{m}^{\beta}} |\mathbf{R}|^{a-(m-1)\beta/2-1}
    q_{\kappa}(\mathbf{R})|\mathbf{I}_{m} + \mathbf{R}|^{-(a+b)} q_{-(\kappa+\tau)}(\mathbf{I}_{m} + \mathbf{R})
    (d\mathbf{R})
$$
$$
   = \frac{\Gamma_{m}^{\beta}[a,\kappa] \Gamma_{m}^{\beta}[b,\tau]}{\Gamma_{m}^{\beta}[a+b,
    \kappa+\tau]},\hspace{8.5cm}
$$
where $\kappa = (k_{1}, k_{2}, \dots, k_{m}) \in \Re^{m}$, $\tau = (t_{1}, t_{2}, \dots, t_{m}) \in
\Re^{m}$, Re$(a)> (m-1)\beta/2-k_{m}$ and Re$(b)> (m-1)\beta/2-t_{m}$. Similarly is defined the
\emph{generalised k-beta function} as, see \citet{dg:15b},
$$
  \mathcal{B}_{m}^{\beta}[a,-\kappa;b, -\tau]  \hspace{10cm}
$$
$$
    \hspace{.1cm}=\int_{\mathbf{0}<\mathbf{S}<\mathbf{I}_{m}}
    |\mathbf{S}|^{a-(m-1)\beta/2-1} q_{\kappa}(\mathbf{S}^{-1})|\mathbf{I}_{m} - \mathbf{S}|^{b-(m-1)\beta/2-1}
    q_{\tau}\left((\mathbf{I}_{m} - \mathbf{S})^{-1}\right)(d\mathbf{S})
$$
$$
  =\int_{\mathbf{R} \in\mathfrak{P}_{m}^{\beta}} |\mathbf{R}|^{a-(m-1)\beta/2-1}
    q_{\kappa}(\mathbf{R}^{-1})|\mathbf{I}_{m} + \mathbf{R}|^{-(a+b)} q_{-(\kappa+\tau)}
    \left((\mathbf{I}_{m} + \mathbf{R})^{-1}\right)(d\mathbf{R})\hspace{1cm}
$$
$$
    = \frac{\Gamma_{m}^{\beta}[a,-\kappa] \Gamma_{m}^{\beta}[b,-\tau]}{\Gamma_{m}^{\beta}[a+b,
    -\kappa-\tau]},\hspace{8.5cm}
$$
where $\kappa = (k_{1}, k_{2}, \dots, k_{m}) \in \Re^{m}$, $\tau = (t_{1}, t_{2}, \dots, t_{m}) \in
\Re^{m}$, Re$(a)> (m-1)\beta/2+k_{1}$ and Re$(b)> (m-1)\beta/2+t_{1}$.

Finally, the following Jacobians involving the $\beta$ parameter, reflects the generalised power of the
algebraic technique; the can be seen as extensions of the full derived and unconnected results in the
real, complex or quaternion cases, see \citet{fk:94} and \citet{dggj:11}. These results are the base for
several matrix and matric variate generalised analysis.

\begin{proposition}\label{lemlt}
Let $\mathbf{X}$ and $\mathbf{Y} \in \mathfrak{L}_{n,m}^{\beta}$  be matrices of functionally independent
variables, and let $\mathbf{Y} = \mathbf{AXB} + \mathbf{C}$, where $\mathbf{A} \in
\mathfrak{L}_{n,n}^{\beta}$, $\mathbf{B} \in \mathfrak{L}_{m,m}^{\beta}$ and $\mathbf{C} \in
\mathfrak{L}_{n,m}^{\beta}$ are constant matrices. Then
\begin{equation}\label{lt}
    (d\mathbf{Y}) = |\mathbf{A}^{*}\mathbf{A}|^{m\beta/2} |\mathbf{B}^{*}\mathbf{B}|^{
    mn\beta/2}(d\mathbf{X}).
\end{equation}
\end{proposition}

\begin{proposition}\label{lemhlt}
Let $\mathbf{X}$ and $\mathbf{Y} \in \mathfrak{S}_{m}^{\beta}$ be matrices of functionally independent
variables, and let $\mathbf{Y} = \mathbf{AXA^{*}} + \mathbf{C}$, where $\mathbf{A} \in
\mathfrak{L}_{m,m}^{\beta}$ and $\mathbf{C} \in \mathfrak{S}_{m}^{\beta}$ are constant matrices. Then
\begin{equation}\label{hlt}
    (d\mathbf{Y}) = |\mathbf{A}^{*}\mathbf{A}|^{(m-1)\beta/2+1} (d\mathbf{X}).
\end{equation}
\end{proposition}

\begin{proposition}\label{lemi}
Let $\mathbf{S} \in \mathfrak{P}_{m}^{\beta}.$ Then, ignoring the sign, if $\mathbf{Y} = \mathbf{S}^{-1}+
\mathbf{C}$, $\mathbf{C} \in \mathfrak{P}_{m}^{\beta}$ is a matrix of constants,
\begin{equation}\label{i}
    (d\mathbf{Y}) = |\mathbf{S}|^{-\beta(m - 1) - 2}(d\mathbf{S}).
\end{equation}
\end{proposition}

\begin{proposition}\label{lemW}
Let $\mathbf{X} \in \mathfrak{L}_{n,m}^{\beta}$  be matrix of functionally independent variables, and
write $\mathbf{X}=\mathbf{V}_{1}\mathbf{T}$, where $\mathbf{V}_{1} \in {\mathcal V}_{m,n}^{\beta}$ and
$\mathbf{T}\in \mathfrak{T}_{U}^{\beta}(m)$ with positive diagonal elements. Define $\mathbf{S} =
\mathbf{X}^{*}\mathbf{X} \in \mathfrak{P}_{m}^{\beta}.$ Then
\begin{equation}\label{w}
    (d\mathbf{X}) = 2^{-m} |\mathbf{S}|^{\beta(n - m + 1)/2 - 1}
    (d\mathbf{S})(\mathbf{V}_{1}^{*}d\mathbf{V}_{1}),
\end{equation}
\end{proposition}

\section{Matricvariate $T$-Riesz distrinution}\label{sec3}

In this section two versions of the matricvariate $T$-Riesz distribution and the corresponding
generalised beta type II distributions are obtained.

A detailed discussion of Riesz distribution may be found in \citet{hl:01} and \citet{dg:15a}. In addition
the Kotz-Riesz distribution is studied in detail in \citet{dg:15c}. For your convenience, we adhere to
standard notation stated in \citet{dg:15a, dg:15b} and consider the following three definitions of
Kotz-Riesz and Riesz and beta-Riesz type I distributions.

From  \citet{dg:15c}.
\begin{definition}\label{defKR}
Let $\boldsymbol{\Sigma} \in \boldsymbol{\Phi}_{m}^{\beta}$, $\boldsymbol{\Theta} \in
\boldsymbol{\Phi}_{n}^{\beta}$, $\boldsymbol{\mu} \in \mathfrak{L}^{\beta}_{n,m}$ and  $\kappa = (k_{1},
k_{2}, \dots, k_{m}) \in \Re^{m}$. And let $\mathbf{Y} \in \mathfrak{L}^{\beta}_{n,m}$ and
$\u(\mathbf{B}) \in \mathfrak{T}_{U}^{\beta}(n)$, such that $\mathbf{B} =
\u(\mathbf{B})^{*}\u(\mathbf{B})$ is the Cholesky decomposition of $\mathbf{B} \in
\mathfrak{S}_{m}^{\beta}$.
\begin{enumerate}
  \item Then it is said that $\mathbf{Y}$ has a \textit{Kotz-Riesz distribution of type I} and its density function is
  $$
    \frac{\beta^{mn\beta/2+\sum_{i = 1}^{m}k_{i}}\Gamma_{m}^{\beta}[n\beta/2]}{\pi^{mn\beta/2}\Gamma_{m}^{\beta}[n\beta/2,\kappa]
    |\boldsymbol{\Sigma}|^{n\beta/2}|\boldsymbol{\Theta}|^{m\beta/2}}\hspace{4cm}
  $$
  $$\hspace{1cm}
    \times \etr\left\{- \beta\tr \left [\boldsymbol{\Sigma}^{-1} (\mathbf{Y} - \boldsymbol{\mu})^{*}
    \boldsymbol{\Theta}^{-1}(\mathbf{Y} - \boldsymbol{\mu})\right ]\right\}
  $$
  \begin{equation}\label{dfEKR1}\hspace{3.1cm}
    \times q_{\kappa}\left [\u(\boldsymbol{\Sigma})^{*-1} (\mathbf{Y} - \boldsymbol{\mu})^{*}
    \boldsymbol{\Theta}^{-1}(\mathbf{Y} - \boldsymbol{\mu})\u(\boldsymbol{\Sigma})^{-1}\right ](d\mathbf{Y})
  \end{equation}
  with $\re(n\beta/2) > (m-1)\beta/2 - k_{m}$;  denoting this fact as
  $$
    \mathbf{Y} \sim \mathcal{KR}^{\beta, I}_{n \times m}
    (\kappa,\boldsymbol{\mu}, \boldsymbol{\Theta}, \boldsymbol{\Sigma}).
  $$
  \item Then it is said that $\mathbf{Y}$ has a \textit{Kotz-Riesz distribution of type II} and its density function is
  $$
    \frac{\beta^{mn\beta/2-\sum_{i = 1}^{m}k_{i}}\Gamma_{m}^{\beta}[n\beta/2]}{\pi^{mn\beta/2}\Gamma_{m}^{\beta}[n\beta/2,-\kappa]
    |\boldsymbol{\Sigma}|^{n\beta/2}|\boldsymbol{\Theta}|^{m\beta/2}}\hspace{4cm}
  $$
  $$
    \times \etr\left\{- \beta\tr \left [\boldsymbol{\Sigma}^{-1} (\mathbf{Y} - \boldsymbol{\mu})^{*}
    \boldsymbol{\Theta}^{-1}(\mathbf{Y} - \boldsymbol{\mu})\right ]\right\}
  $$
  \begin{equation}\label{dfEKR2}\hspace{2.5cm}
    \times q_{\kappa}\left [\left(\u(\boldsymbol{\Sigma})^{*-1} (\mathbf{Y} - \boldsymbol{\mu})^{*}
    \boldsymbol{\Theta}^{-1}(\mathbf{Y} - \boldsymbol{\mu})\u(\boldsymbol{\Sigma})^{-1/2}\right)^{-1}\right ](d\mathbf{Y})
  \end{equation}
  with $\re(n\beta/2) > (m-1)\beta/2 + k_{1}$;  denoting this fact as
  $$
    \mathbf{Y} \sim \mathcal{KR}^{\beta, II}_{n \times m}
    (\kappa,\boldsymbol{\mu}, \boldsymbol{\Theta}, \boldsymbol{\Sigma}).
  $$
\end{enumerate}
\end{definition}

From \citet{hl:01} and \citet{dg:15a}.
\begin{definition}\label{defR}
Let $\mathbf{\Xi} \in \mathbf{\Phi}_{m}^{\beta}$ and  $\kappa = (k_{1}, k_{2}, \dots, k_{m}) \in
\Re^{m}$.
\begin{enumerate}
  \item Then it is said that $\mathbf{V}$ has a \textit{Riesz distribution of type I} if its density function is
  \begin{equation}\label{dfR1}
    \frac{\beta^{am+\sum_{i = 1}^{m}k_{i}}}{\Gamma_{m}^{\beta}[a,\kappa] |\mathbf{\Xi}|^{a}q_{\kappa}(\mathbf{\Xi})}
    \etr\{-\beta\mathbf{\Xi}^{-1}\mathbf{V}\}|\mathbf{V}|^{a-(m-1)\beta/2 - 1}
    q_{\kappa}(\mathbf{V})(d\mathbf{V})
  \end{equation}
  for $\mathbf{V} \in \mathfrak{P}_{m}^{\beta}$ and $\re(a) \geq (m-1)\beta/2 - k_{m}$;
  denoting this fact as $\mathbf{V} \sim \mathcal{R}^{\beta, I}_{m}(a,\kappa,
  \mathbf{\Xi})$.
  \item Then it is said that $\mathbf{V}$ has a \textit{Riesz distribution of type II} if its density function is
  \begin{equation}\label{dfR2}
     \frac{\beta^{am-\sum_{i = 1}^{m}k_{i}}}{\Gamma_{m}^{\beta}[a,-\kappa]
   |\mathbf{\Xi}|^{a}q_{\kappa}(\mathbf{\Xi}^{-1})}\etr\{-\beta\mathbf{\Xi}^{-1}\mathbf{V}\}
  |\mathbf{V}|^{a-(m-1)\beta/2 - 1} q_{\kappa}(\mathbf{V}^{-1}) (d\mathbf{V})
  \end{equation}
  for $\mathbf{V} \in \mathfrak{P}_{m}^{\beta}$ and $\re(a) > (m-1)\beta/2 + k_{1}$;
  denoting this fact as $\mathbf{V} \sim \mathcal{R}^{\beta, II}_{m}(a,\kappa,
  \mathbf{\Xi})$.
\end{enumerate}
\end{definition}

Now we propose the definitions of Pearson type II-Riesz (see \citet{dgcl:15}) and T-Riesz distributions.

\begin{definition}\label{defPR}
Let $\kappa = (k_{1}, k_{2}, \dots, k_{m}) \in \Re^{m}$, and $\tau = (t_{1}, t_{2}, \dots, t_{m}) \in
\Re^{m}$. Also define $\mathbf{R}\in \mathfrak{L}_{n,m}^{\beta}$ as
$$
  \mathbf{R} = \mathbf{X}\u(\mathbf{U}_{1} + \mathbf{X}^{*}\mathbf{X})^{-1},
$$
where $\u(\mathbf{U}_{1} + \mathbf{X}^{*}\mathbf{X}) \in \mathfrak{T}_{U}^{\beta}(m)$ is such that
$\mathbf{U}=\u(\mathbf{U}_{1} + \mathbf{X}^{*}\mathbf{X})^{*}\u(\mathbf{U}_{1} +
\mathbf{X}^{*}\mathbf{X})$  is the Cholesky decomposition of $\mathbf{U}$,
\begin{enumerate}
  \item with $\mathbf{U}_{1}\sim \mathcal{R}_{m}^{\beta,I}(\nu\beta/2,
    \kappa,\mathbf{I}_{m})$,  $\re(\nu\beta/2)> (m-1)\beta/2-k_{m}$; independent of $\mathbf{X}
    \sim \mathcal{KR}_{n \times m}^{\beta,I}(\tau,\mathbf{0}, \mathbf{I}_{n}, \mathbf{I}_{m})$,
    $\re(n\beta/2)> (m-1)\beta/2-t_{m}$. The random matrix $\mathbf{R}  \in
    \mathfrak{L}^{\beta}_{n,m}$ is said to have the \textit{matricvariate
  Pearson type II-Riesz distribution type I} if its density is given by
  \begin{equation}\label{DP2RI}
        \frac{\Gamma_{m}^{\beta}[n\beta/2]\quad\left|\mathbf{I}_{m} - \mathbf{R}^{*}\mathbf{R}\right|^{(\nu-m+1)\beta/2-1}}
        {\pi^{mn\beta/2}\mathcal{B}_{m}^{\beta}[\nu\beta /2,\kappa;n\beta/2,\tau]} \
        q_{\kappa}\left(\mathbf{I}_{m} - \mathbf{R}^{*}\mathbf{R}\right)
        q_{\tau}\left(\mathbf{R}^{*}\mathbf{R}\right)(d\mathbf{R}),
  \end{equation}
  where $\mathbf{I}_{m} - \mathbf{R}^{*}\mathbf{R} \in \mathfrak{P}^{\beta}_{m}$.
  This fact is denoted as
  $$
     \mathbf{R} \sim \mathcal{P_{II}R}_{n \times m}^{\beta, I}
    (\nu,\kappa,\tau,\mathbf{0}, \mathbf{I}_{n},\mathbf{I}_{m}).
  $$

  \item with $\mathbf{U}_{1}\sim \mathcal{R}_{m}^{\beta,II}(\nu\beta/2,
    \kappa,\mathbf{I}_{m})$,  $\re(\nu\beta/2)> (m-1)\beta/2+k_{1}$; independent of $\mathbf{X}
    \sim \mathcal{KR}_{n \times m}^{\beta,II}(\tau,\mathbf{0}, \mathbf{I}_{n}, \mathbf{I}_{m})$,
    $\re(n\beta/2)> (m-1)\beta/2+t_{1}$. The random matrix $\mathbf{R}  \in \mathfrak{L}^{\beta}_{n,m}$
    is said to have the \textit{matricvariate Pearson type II-Riesz distribution type II} if its density is given by
    \begin{equation*}
      \frac{\Gamma_{m}^{\beta}[n\beta/2]\quad\left|\mathbf{I}_{m} - \mathbf{R}^{*}\mathbf{R}\right|^{(\nu-m+1)\beta/2-1}}
      {\pi^{mn\beta/2}\mathcal{B}_{m}^{\beta}[\nu\beta /2,-\kappa;n\beta/2,-\tau]} \
      q_{\kappa}\left[\left(\mathbf{I}_{m} - \mathbf{R}^{*}\mathbf{R}\right)^{-1}\right]\hspace{2cm}
    \end{equation*}
    \begin{equation}\label{DP2RII}
      \hspace{7cm}\times \ q_{\tau}\left[\left(\mathbf{R}^{*}\mathbf{R}\right)^{-1}\right](d\mathbf{R}),
    \end{equation}
  where $\mathbf{I}_{m} - \mathbf{R}^{*}\mathbf{R} \in \mathfrak{P}^{\beta}_{m}$.
  This fact is denoted as
  $$
    \mathbf{R} \sim \mathcal{P_{II}R}_{n \times m}^{\beta, II}
    (\nu,\kappa,\tau,\mathbf{0}, \mathbf{I}_{n},\mathbf{I}_{m}).
  $$
\end{enumerate}
\end{definition}

\begin{definition}\label{defTRII}
Let $\kappa = (k_{1}, k_{2}, \dots, k_{m}) \in \Re^{m}$, and $\tau = (t_{1}, t_{2}, \dots, t_{m}) \in
\Re^{m}$.
\begin{enumerate}
  \item An random matrix $\mathbf{T}  \in \mathfrak{L}^{\beta}_{n,m}$  is said to have the \textit{matricvariate
  T-Riesz distribution type I} if its density is given by
  \begin{equation}\label{DTRI}
        \frac{\Gamma_{m}^{\beta}[n\beta/2]\quad\left|\mathbf{I}_{m} + \mathbf{T}^{*}\mathbf{T}\right|^{-(\nu+n)\beta /2}}
        {\pi^{mn\beta/2}\mathcal{B}_{m}^{\beta}[\nu\beta /2,\kappa;n\beta/2,\tau]} \
        q_{-\kappa-\tau}\left(\mathbf{I}_{m} + \mathbf{T}^{*}\mathbf{T}\right)
        q_{\tau}\left(\mathbf{T}^{*}\mathbf{T}\right)(d\mathbf{T}),
  \end{equation}
  where $\re(\nu\beta/2)> (m-1)\beta/2-k_{m}$, $\re(n\beta/2)> (m-1)\beta/2-t_{m}$.
  This fact is denoted as $\mathbf{T} \sim \mathcal{TR}_{n \times m}^{\beta, I}
  (\nu,\kappa,\tau,\mathbf{0}, \mathbf{I}_{n},\mathbf{I}_{m})$.

  \item An random matrix $\mathbf{T}  \in \mathfrak{L}^{\beta}_{n,m}$  is said to have the \textit{matricvariate
  T-Riesz distribution type II} if its density is given by
    \begin{equation*}
      \frac{\Gamma_{m}^{\beta}[n\beta/2]\quad\left|\mathbf{I}_{m} + \mathbf{T}^{*}\mathbf{T}\right|^{-(\nu+n)\beta /2}}
      {\pi^{mn\beta/2}\mathcal{B}_{m}^{\beta}[\nu\beta /2,-\kappa;n\beta/2,-\tau]} \
      q_{-\kappa-\tau}\left[\left(\mathbf{I}_{m} + \mathbf{T}^{*}\mathbf{T}\right)^{-1}\right]\hspace{2cm}
    \end{equation*}
    \begin{equation}\label{DTRII}
      \hspace{7cm}\times \ q_{\tau}\left[\left(\mathbf{T}^{*}\mathbf{T}\right)^{-1}\right](d\mathbf{R}),
    \end{equation}
  where $\re(\nu\beta/2)> (m-1)\beta/2+k_{1}$, $\re(n\beta/2)> (m-1)\beta/2+t_{1}$.
  This fact is denoted as $\mathbf{T} \sim \mathcal{TR}_{n \times m}^{\beta, II}
  (\nu,\kappa,\tau,\mathbf{0}, \mathbf{I}_{n},\mathbf{I}_{m})$.
\end{enumerate}
\end{definition}

\begin{theorem}\label{teo1}
Assume that $\mathbf{R} \sim \mathcal{P_{II}R}_{n \times m}^{\beta, I} (\nu,\kappa,\tau,\mathbf{0},
\mathbf{I}_{n},\mathbf{I}_{m})$. Then, if $\mathbf{T} = \mathbf{R}
\u(\mathbf{I}_{m}-\mathbf{R}^{*}\mathbf{R})^{-1}$, we have that $\mathbf{T} \sim \mathcal{TR}_{n \times
m}^{\beta, I} (\nu,\kappa,\tau,\mathbf{0}, \mathbf{I}_{n},\mathbf{I}_{m})$. Where
$\u(\mathbf{I}_{m}-\mathbf{R}^{*}\mathbf{R}) \in \mathfrak{T}_{U}^{\beta}(m)$ is such that
$(\mathbf{I}_{m}-\mathbf{R}^{*}\mathbf{R}) =
\u(\mathbf{I}_{m}-\mathbf{R}^{*}\mathbf{R})^{*}\u(\mathbf{I}_{m}-\mathbf{R}^{*}\mathbf{R})$ is the
Cholesky decomposition of $(\mathbf{I}_{m}-\mathbf{R}^{*}\mathbf{R})$.
\end{theorem}
\begin{proof}
From Definition \ref{defTRII}.1 the density of $\mathbf{R}$ is
\begin{equation}\label{DR1Ip}
        \frac{\Gamma_{m}^{\beta}[n\beta/2]\quad\left|\mathbf{I}_{m} -
        \mathbf{R}^{*}\mathbf{R}\right|^{\nu\beta/2-p}}
        {\pi^{mn\beta/2}\mathcal{B}_{m}^{\beta}[\nu\beta /2,\kappa;n\beta/2,\tau]} \
        q_{\kappa}\left(\mathbf{I}_{m} - \mathbf{R}^{*}\mathbf{R}\right)
        q_{\tau}\left(\mathbf{R}^{*}\mathbf{R}\right)(d\mathbf{R}),
  \end{equation}
where $p = (m-1)\beta/2+1$. Now, observe that
\begin{eqnarray*}
  \mathbf{T}^{*}\mathbf{T} &=& \u(\mathbf{I}_{m}-\mathbf{R}^{*}\mathbf{R})^{*-1}\mathbf{R}^{*}\mathbf{R}
        \u(\mathbf{I}_{m}-\mathbf{R}^{*}\mathbf{R})^{-1}\\
    &=&
    \u(\mathbf{I}_{m}-\mathbf{R}^{*}\mathbf{R})^{*-1}(\mathbf{I}_{m}-(\mathbf{I}_{m}-\mathbf{R}^{*}\mathbf{R}))
        \u(\mathbf{I}_{m}-\mathbf{R}^{*}\mathbf{R})^{-1}\\
    &=& \u(\mathbf{I}_{m}-\mathbf{R}^{*}\mathbf{R})^{*-1}
    \u(\mathbf{I}_{m}-\mathbf{R}^{*}\mathbf{R})^{-1} - \mathbf{I}_{m},
\end{eqnarray*}
then $\mathbf{I}_{m} + \mathbf{T}^{*}\mathbf{T} = \u(\mathbf{I}_{m}-\mathbf{R}^{*}\mathbf{R})^{*-1}
\u(\mathbf{I}_{m}-\mathbf{R}^{*}\mathbf{R})^{-1}$. Then, by lemmas \ref{lemlt} and \ref{lemi},
\begin{equation}\label{j1}
    (d\mathbf{T}) = |\mathbf{I}_{m} + \mathbf{T}^{*}\mathbf{T}|^{n\beta/2+p} (d\mathbf{R})
\end{equation}
Also note that
\begin{eqnarray*}
  q_{\tau}(\mathbf{T}^{*}\mathbf{T}) &=& q_{\tau}\left(\u(\mathbf{I}_{m}-\mathbf{R}^{*}\mathbf{R})^{*-1}\mathbf{R}^{*}\mathbf{R}
        \u(\mathbf{I}_{m}-\mathbf{R}^{*}\mathbf{R})^{-1}\right)\\
   &=& q_{-\tau}(\mathbf{I}_{m}-\mathbf{R}^{*}\mathbf{R})q_{\tau}(\mathbf{R}^{*}\mathbf{R}),
\end{eqnarray*}
and
%\begin{small}
\begin{eqnarray*}
  |\mathbf{I}_{m} + \mathbf{T}^{*}\mathbf{T}|^{\nu\beta/2-p}q_{\kappa}(\mathbf{I}_{m} + \mathbf{T}^{*}\mathbf{T}) &=&
  |\mathbf{I}_{m} + \u(\mathbf{I}_{m}-\mathbf{R}^{*}\mathbf{R})^{*-1}\mathbf{R}^{*}
  \mathbf{R}\u(\mathbf{I}_{m}-\mathbf{R}^{*}\mathbf{R})^{-1}|^{\nu\beta/2-p}\\
   & & \times \ q_{\kappa}\left(\mathbf{I}_{m} +
        \u(\mathbf{I}_{m}-\mathbf{R}^{*}\mathbf{R})^{*-1}\mathbf{R}^{*}\mathbf{R}
        \u(\mathbf{I}_{m}-\mathbf{R}^{*}\mathbf{R})^{-1}\right) \\
   &=&  |\mathbf{I}_{m}-\mathbf{R}^{*}\mathbf{R}|^{-\nu\beta/2-p}q_{-\kappa}(\mathbf{I}_{m}-\mathbf{R}^{*}\mathbf{R}).
\end{eqnarray*}
%\end{small}
From where
\begin{equation}\label{i1}
    |\mathbf{I}_{m}-\mathbf{R}^{*}\mathbf{R}|^{\nu\beta/2+p}q_{\kappa}(\mathbf{I}_{m}-\mathbf{R}^{*}\mathbf{R})
  = |\mathbf{I}_{m} + \mathbf{T}^{*}\mathbf{T}|^{-\nu\beta/2pp}q_{-\kappa}(\mathbf{I}_{m} + \mathbf{T}^{*}\mathbf{T})
\end{equation}
and
\begin{equation}\label{i2}
    q_{\tau}(\mathbf{R}^{*}\mathbf{R})= q_{\tau}(\mathbf{T}^{*}\mathbf{T})q_{-\tau}(\mathbf{I}_{m} +
    \mathbf{T}^{*}\mathbf{T}).
\end{equation}
Substituting (\ref{i1}), (\ref{i2}) and (\ref{j1}) in (\ref{DR1Ip}), the density desired is obtained.
\end{proof}

Similarly, for the matricvariate  T-Riesz distribution type II we have:

\begin{theorem}\label{teo11}
Let $\u(\mathbf{I}_{m}-\mathbf{R}^{*}\mathbf{R}) \in \mathfrak{T}_{U}^{\beta}(m)$, it is such that
$(\mathbf{I}_{m}-\mathbf{R}^{*}\mathbf{R}) = \u(\mathbf{I}_{m}-\mathbf{R}^{*}
\mathbf{R})^{*}\u(\mathbf{I}_{m} - \mathbf{R}^{*}\mathbf{R})$ is the Cholesky decomposition of
$(\mathbf{I}_{m}-\mathbf{R}^{*}\mathbf{R})$. In addition, let $\mathbf{R} \sim \mathcal{P_{II}R}_{n
\times m}^{\beta, II} (\nu,\kappa,\tau,\mathbf{0}, \mathbf{I}_{n},\mathbf{I}_{m})$. Then, if $\mathbf{T}
= \mathbf{R} \u(\mathbf{I}_{m}-\mathbf{R}^{*}\mathbf{R})^{-1}$, we obtain that
$$
  \mathbf{T} \sim \mathcal{TR}_{n \times m}^{\beta, II} (\nu,\kappa,\tau,\mathbf{0}, \mathbf{I}_{n},\mathbf{I}_{m}).
$$
\end{theorem}
\begin{proof}
The proof is similar to that given for Theorem \ref{teo1}.
\end{proof}

Next we obtain the matricvariate T-Riesz distributions in terms of matricvariate Kotz-Riesz and Riesz
distributions.

\begin{theorem}\label{teo31} Let $\kappa = (k_{1}, k_{2}, \dots, k_{m})\in \Re^{m}$, and $\tau = (t_{1},
t_{2}, \dots, t_{m})\in \Re^{m}$. Also define $\mathbf{T}\in \mathfrak{L}_{n,m}^{\beta}$ as
$$
  \mathbf{T} = \mathbf{X}\u(\mathbf{U})^{-1}
$$
where $\u(\mathbf{U}) \in \mathfrak{T}_{U}^{\beta}(m)$ is such that $\mathbf{U} =
\u(\mathbf{U})^{*}\u(\mathbf{U})$ is the Cholesky decomposition of $\mathbf{U }\sim
\mathcal{R}_{m}^{\beta,I}(\nu\beta/2, \kappa,\mathbf{I}_{m})$, $\re(\nu\beta/2)> (m-1)\beta/2-k_{m}$
independent of $\mathbf{X} \sim \mathcal{KR}_{n \times m}^{\beta,I}(\tau, \mathbf{0},
\mathbf{I}_{n},\mathbf{I}_{m})$, $\re(n\beta/2)> (m-1)\beta/2-t_{m}$. Then, $\mathbf{T} \sim
\mathcal{TR}_{n \times m}^{\beta,I}(\nu,\kappa, \tau, \mathbf{0}, \mathbf{I}_{n}, \mathbf{I}_{m})$.
\end{theorem}
\begin{proof}
From  theorems \ref{teo1} and \ref{teo11}, we know that $\mathbf{T} = \mathbf{R}
\u(\mathbf{I}_{m}-\mathbf{R}^{*}\mathbf{R})^{-1}$. The desired result is obtained if we can to proof that
$$
  \mathbf{T} = \mathbf{R}\u(\mathbf{I}_{m}-\mathbf{R}^{*}\mathbf{R})^{-1} =
  \mathbf{X}\u(\mathbf{U})^{-1}.
$$
With this goal in mind, from Definition \ref{defPR} we know that $\mathbf{U} + \mathbf{X}^{*}\mathbf{X} =
\u(\mathbf{U} + \mathbf{X}^{*}\mathbf{X})^{*}\u(\mathbf{U} + \mathbf{X}^{*}\mathbf{X})$, with
$\u(\mathbf{U} + \mathbf{X}^{*}\mathbf{X}) \in \mathfrak{T}_{U}^{\beta}(m)$ then $\mathbf{R} = \mathbf{X}
\u(\mathbf{U} + \mathbf{X}^{*}\mathbf{X})^{-1}$. Hence
\begin{equation}\label{eqx}
    \mathbf{X} = \mathbf{R}\u(\mathbf{U} + \mathbf{X}^{*}\mathbf{X})
\end{equation}
and $\mathbf{R}^{*}\mathbf{R} = \u(\mathbf{U} + \mathbf{X}^{*}\mathbf{X})^{*-1}\mathbf{X} ^{*}\mathbf{X}
\u(\mathbf{U} + \mathbf{X}^{*}\mathbf{X})^{-1}$. By hypothesis $\mathbf{U} =
\u(\mathbf{U})^{*}\u(\mathbf{U})$, where $\u(\mathbf{U}) \in \mathfrak{T}_{U}^{\beta}(m)$. Therefore
\begin{eqnarray*}
  \mathbf{I}_{m} - \mathbf{R}^{*}\mathbf{R} &=& \mathbf{I}_{m} - \u(\mathbf{U} + \mathbf{X}^{*}\mathbf{X})^{*-1}
    \mathbf{X} ^{*}\mathbf{X} \u(\mathbf{U} + \mathbf{X}^{*}\mathbf{X})^{-1}\\
   &=& \u(\mathbf{U} + \mathbf{X}^{*}\mathbf{X})^{*-1} (\u(\mathbf{U} + \mathbf{X}^{*}\mathbf{X})^{*}\u(\mathbf{U} +
    \mathbf{X}^{*}\mathbf{X})- \mathbf{X} ^{*}\mathbf{X}) \u(\mathbf{U} + \mathbf{X}^{*}\mathbf{X})^{-1} \\
   &=& \u(\mathbf{U} + \mathbf{X}^{*}\mathbf{X})^{*-1}\mathbf{U}\u(\mathbf{U} + \mathbf{X}^{*}\mathbf{X})^{-1} \\
   &=& \u(\mathbf{U} + \mathbf{X}^{*}\mathbf{X})^{*-1}\u(\mathbf{U})^{*}\u(\mathbf{U})\u(\mathbf{U} + \mathbf{X}^{*}
    \mathbf{X})^{-1}  \\
   &=& \left(\u(\mathbf{U})\u(\mathbf{U} + \mathbf{X}^{*}\mathbf{X})^{-1}\right)^{*}\left(\u(\mathbf{U})\u(\mathbf{U} +
    \mathbf{X}^{*}\mathbf{X})^{-1}\right).
\end{eqnarray*}
Thus, by uniqueness of the Cholesky decomposition, $\u(\mathbf{U})\u(\mathbf{U} + \mathbf{X}^{*}
\mathbf{X})^{-1} = \u(\mathbf{I}_{m} - \mathbf{R}^{*}\mathbf{R}) \in \mathfrak{T}_{U}^{\beta}(m)$. Hence
\begin{eqnarray*}
  \mathbf{T} = \mathbf{R} \u(\mathbf{I}_{m}-\mathbf{R}^{*}\mathbf{R})^{-1} &=& \mathbf{R}\left(\u(\mathbf{U})\u(\mathbf{U} + \mathbf{X}^{*}
\mathbf{X})^{-1}\right)^{-1}\\
   &=& \mathbf{R}\u(\mathbf{U} + \mathbf{X}^{*}\mathbf{X})\u(\mathbf{U})^{-1},  \quad \mbox{by (\ref{eqx}),}  \\
   &=& \mathbf{X}\u(\mathbf{U})^{-1},
\end{eqnarray*}
which concludes the proof.
\end{proof}

Similar to Theorem \ref{teo31}, now assuming distributions Kotz-Riesz and Riesz type II  the following
result is obtained.

\begin{theorem}\label{teo32}
Let $\kappa = (k_{1}, k_{2}, \dots, k_{m})\in \Re^{m}$, and $\tau = (t_{1}, t_{2}, \dots, t_{m})\in
\Re^{m}$. Also define $\mathbf{T}\in \mathfrak{L}_{m,n}^{\beta}$ as
$$
  \mathbf{T} = \mathbf{X}\u(\mathbf{U})^{-1}
$$
where $\u(\mathbf{U}) \in \mathfrak{T}_{U}^{\beta}(m)$ is such that $\mathbf{U} =
\u(\mathbf{U})^{*}\u(\mathbf{U})$ is the Cholesky decomposition of $\mathbf{U }\sim
\mathcal{R}_{m}^{\beta,II}(\nu\beta/2, \kappa,\mathbf{I}_{m})$, $\re(\nu\beta/2)> (m-1)\beta/2+k_{1}$
independent of $\mathbf{X} \sim \mathcal{KR}_{n \times m}^{\beta,II}(\tau, \mathbf{0},
\mathbf{I}_{n},\mathbf{I}_{m})$, $\re(n\beta/2)> (m-1)\beta/2+t_{1}$. Then, $ \mathbf{T} \sim
\mathcal{TR}_{n \times m}^{\beta,II}(\nu,\kappa, \tau, \mathbf{0}, \mathbf{I}_{n}, \mathbf{I}_{m})$.
\end{theorem}
\begin{proof}
The proof is the same that given for Theorem \ref{teo31}.
\end{proof}

\begin{corollary}\label{corTRI}
Assume that $\mathbf{T} \sim \mathcal{TR}_{n \times m}^{\beta,I}(\nu,\kappa, \tau, \mathbf{0},
\mathbf{I}_{n}, \mathbf{I}_{m})$ and define $\mathbf{S} = \u(\mathbf{\Delta})^{-1} \mathbf{T}
\u(\mathbf{\Pi}) + \boldsymbol{\mu}$, where $\u(\mathbf{\Delta}) \in \mathfrak{T}_{U}^{\beta}(n)$ and
$\u(\mathbf{\Pi}) \in \mathfrak{T}_{U}^{\beta}(m)$ are such that $\mathbf{\Delta}
=\u(\mathbf{\Delta})^{*} \u(\mathbf{\Delta})$ and $\mathbf{\Pi} = \u(\mathbf{\Pi})^{*}\u(\mathbf{\Pi})$
are the Cholesky decomposition of $\mathbf{\Delta}$ and $\mathbf{\Pi}$, respectively. In addition, let
$\boldsymbol{\mu} \in \mathfrak{L}_{n \times m}^{\beta}$, $\mathbf{\Delta} \in \mathfrak{P}_{n}^{\beta}$
and $\mathbf{\Pi} \in \mathfrak{P}_{m}^{\beta}$, constant matrices. Then the density of $\mathbf{S}$ is
$$
  \frac{\Gamma_{m}^{\beta}[n\beta/2]|\mathbf{\Pi}|^{\nu\beta/2}q_{\kappa}(\mathbf{\Pi})
  |\mathbf{\Delta}|^{m\beta/2}\quad\left|\mathbf{\Pi} + (\mathbf{S} - \boldsymbol{\mu})^{*}
  \mathbf{\Delta}(\mathbf{S} - \boldsymbol{\mu})\right|^{-(\nu+n)\beta /2}}
  {\pi^{mn\beta/2}\mathcal{B}_{m}^{\beta}[\nu\beta /2,\kappa;n\beta/2,\tau]}
$$
$$
  \times \
  q_{-\kappa-\tau}\left(\mathbf{\Pi} + (\mathbf{S} - \boldsymbol{\mu})^{*}\mathbf{\Delta}(\mathbf{S} - \boldsymbol{\mu})\right)
  q_{\tau}\left((\mathbf{S} - \boldsymbol{\mu})^{*}\mathbf{\Delta}(\mathbf{S} - \boldsymbol{\mu})\right)(d\mathbf{S}),
$$
where $\mathbf{\Pi} + (\mathbf{S} - \boldsymbol{\mu})^{*}\mathbf{\Delta}(\mathbf{S} - \boldsymbol{\mu})
\in \mathfrak{P}_{m}^{\beta}$, $\re(\nu\beta/2)> (m-1)\beta/2-k_{m}$, $\re(n\beta/2)>
(m-1)\beta/2-t_{m}$. This fact is denoted as $\mathbf{S} \sim \mathcal{TR}_{n \times
m}^{\beta,I}(\nu,\kappa, \tau, \boldsymbol{\mu}, \mathbf{\Delta}, \mathbf{\Pi})$.
\end{corollary}
\begin{proof}
The proof is follows making in (\ref{DTRI}) the change of variable $\mathbf{S} = \u(\mathbf{\Delta})^{-1}
\mathbf{T} \u(\mathbf{\Pi}) + \boldsymbol{\mu}$. Observing that by Proposition \ref{lemlt},
$$
  (d\mathbf{T}) = |\mathbf{\Delta}|^{m\beta/2}|\mathbf{\Pi}|^{-n\beta/2}(d\mathbf{S}).
$$
\end{proof}

\begin{corollary}\label{corTRII}
Assume that $\mathbf{T} \sim \mathcal{TR}_{n \times m}^{\beta,II}(\nu,\kappa, \tau, \mathbf{0},
\mathbf{I}_{n}, \mathbf{I}_{m})$ and define $\mathbf{S} = \u(\mathbf{\Delta})^{-1} \mathbf{T}
\u(\mathbf{\Pi}) + \boldsymbol{\mu}$, where $\u(\mathbf{\Delta}) \in \mathfrak{T}_{U}^{\beta}(n)$ and
$\u(\mathbf{\Pi}) \in \mathfrak{T}_{U}^{\beta}(m)$ are such that $\mathbf{\Delta}
=\u(\mathbf{\Delta})^{*} \u(\mathbf{\Delta})$ and $\mathbf{\Pi} = \u(\mathbf{\Pi})^{*}\u(\mathbf{\Pi})$
are the Cholesky decomposition of $\mathbf{\Delta}$ and $\mathbf{\Pi}$, respectively. And
$\boldsymbol{\mu} \in \mathfrak{L}_{n \times m}^{\beta}$, $\mathbf{\Delta} \in \mathfrak{P}_{n}^{\beta}$
and $\mathbf{\Pi} \in \mathfrak{P}_{m}^{\beta}$ are constant matrices. Then the density of $\mathbf{S}$
$$
  \frac{\Gamma_{m}^{\beta}[n\beta/2]|\mathbf{\Pi}|^{\nu\beta/2}q_{\kappa}(\mathbf{\Pi}^{-1})
  |\mathbf{\Delta}|^{m\beta/2}\quad\left|\mathbf{\Pi} + (\mathbf{S} - \boldsymbol{\mu})^{*}
  \mathbf{\Delta}(\mathbf{S} - \boldsymbol{\mu})\right|^{-(\nu+n)\beta /2}}
  {\pi^{mn\beta/2}\mathcal{B}_{m}^{\beta}[\nu\beta /2,-\kappa;n\beta/2,-\tau]}
$$
$$
  \times \
  q_{-\kappa-\tau}\left[\left(\mathbf{\Pi} + (\mathbf{S} - \boldsymbol{\mu})^{*}\mathbf{\Delta}(\mathbf{S} -
  \boldsymbol{\mu})\right)^{-1}\right]
  q_{\tau}\left[\left((\mathbf{S} - \boldsymbol{\mu})^{*}\mathbf{\Delta}(\mathbf{S} -
  \boldsymbol{\mu})\right)^{-1}\right](d\mathbf{S}),
$$
where $\mathbf{\Pi} + (\mathbf{S} - \boldsymbol{\mu})^{*}\mathbf{\Delta}(\mathbf{S} - \boldsymbol{\mu})
\in \mathfrak{P}_{m}^{\beta}$, $\re(\nu\beta/2)> (m-1)\beta/2+k_{1}$, $\re(n\beta/2)>
(m-1)\beta/2+t_{1}$. This fact is denoted as $\mathbf{S} \sim \mathcal{TR}_{n \times
m}^{\beta,II}(\nu,\kappa, \tau, \boldsymbol{\mu}, \mathbf{\Delta}, \mathbf{\Pi})$.
\end{corollary}
\begin{proof}
The proof is a verbatim copy of given to Corollary \ref{corTRI}.
\end{proof}

Observe that if theorems \ref{teo1}, \ref{teo11}, \ref{teo31} and \ref{teo32} and corollaries
\ref{corTRI} and \ref{corTRII} are defined $\kappa = (0, \dots,0)$ and $\tau = (0, \dots,0)$, results in
\citet{dggj:12} and \citet{di:67} (with $\beta = 1$) are obtained as particular cases.

\begin{remark}
Alternatively, observe that the random matrix $\mathbf{T}$ can be defined as follows: Let $\kappa_{_{1}}
= k_{11}, k_{12}, \dots, k_{1n})\in \Re^{n}$, and $\tau_{_{1}} = (t_{11}, t_{12}, \dots, t_{1n})\in
\Re^{n}$. Also define $\mathbf{T}\in \mathfrak{L}_{n,m}^{\beta}$ as
$$
  \mathbf{T}_{1} = \u(\mathbf{U}_{1})^{-1}\mathbf{Y}
$$
where $\u(\mathbf{U}_{1})^{*} \in \mathfrak{T}_{U}^{\beta}(n)$ is such that $\mathbf{U}_{1} =
\u(\mathbf{U}_{1})\u(\mathbf{U}_{1})^{*}$ is the Cholesky decomposition of $\mathbf{U}_{1} \sim
\mathcal{R}_{n}^{\beta,I}(a\beta/2, \kappa_{_{1}},\mathbf{I}_{n})$, $\re(a\beta/2)> (n-1)\beta/2-k_{1n}$
independent of $\mathbf{Y} = \mathbf{X}^{*} \sim \mathcal{KR}_{n \times m}^{\beta,I}(\tau_{_{1}},
\mathbf{0}, \mathbf{I}_{n},\mathbf{I}_{m})$, $\re(m\beta/2)> (n-1)\beta/2-t_{1n}$. Then,
$$
  \mathbf{T}_{1} \sim \mathcal{TR}_{n \times m}^{\beta,I}(a\beta/2,\kappa_{_{1}}, \tau_{_{1}}, \mathbf{0},
\mathbf{I}_{n}, \mathbf{I}_{m}).
$$
Then, the corresponding density is obtained from (\ref{DTRI}) making
the following substitutions
\begin{equation}\label{ss}
    \mathbf{T} \rightarrow \mathbf{T}_{1}^{*}, \quad m \rightarrow n, \quad n \rightarrow m, \quad \nu
  \rightarrow a
\end{equation}
and thus, $\kappa \rightarrow \kappa_{1}$, $\tau \rightarrow \tau_{1},$ and $k_{i} \rightarrow k_{1i}
\quad t_{i} \rightarrow t_{1i}$.

Analogously, the distribution of $\mathbf{S}_{1} = \u(\mathbf{\Theta})\mathbf{T}_{1}
\u(\mathbf{\Pi})^{-1} + \boldsymbol{\mu}$, with $\u(\mathbf{\Theta})^{*} \in
\mathfrak{T}_{U}^{\beta}(n)$, and $\u(\mathbf{\Pi})^{*} \in \mathfrak{T}_{U}^{\beta}(m)$ are such that
$\mathbf{\Theta} =\u(\mathbf{\Theta})\u(\mathbf{\Theta})^{*}$ and $\mathbf{\Pi} =
\u(\mathbf{\Pi})\u(\mathbf{\Pi})^{*}$ also is obtained from Corollary \ref{corTRI} making the
substitutions (\ref{ss}). In analogous way the same results assuming a T-Riesz type II are obtained.
\end{remark}

\section{Matricvariate beta-Riesz distribution type II}\label{sec4}

Let $\mathbf{F} \in \mathfrak{P}_{m}^{\beta}$ be defined as $\mathbf{F} = \mathbf{T}^{*}\mathbf{T}$, such
that $\mathbf{T} \sim \mathcal{TR}_{n \times m}^{\beta,I}(\nu,\kappa, \tau, \boldsymbol{0},
\mathbf{I}_{m}, \mathbf{I}_{n})$ with $n \geq m$. Then, under the conditions of Theorem \ref{teo31}, we
have
\begin{eqnarray*}
  \mathbf{F} &=& \u(\mathbf{U})^{*-1}\mathbf{X}^{*}\mathbf{X}\u(\mathbf{U})^{-1} \\
   &=& \u(\mathbf{U})^{*-1}\mathbf{W}\u(\mathbf{U})^{-1}
\end{eqnarray*}
where $\mathbf{W}=\mathbf{X}^{*}\mathbf{X} \sim \mathcal{R}_{m}^{\beta,I}(n\beta/2,
\tau,\mathbf{I}_{m})$, with $\re(n\beta/2) >(m-1)\beta/2-t_{m}$. Thus:
\begin{theorem}\label{teo2}
The density of $\mathbf{F}$ is
\begin{equation}\label{FII1}
    \propto |\mathbf{F}|^{(n-m+1)\beta/2-1}|\mathbf{I}_{m}+\mathbf{F}|^{-(n+\nu)\beta/2}
    q_{-\kappa-\tau}(\mathbf{I}_{m}+\mathbf{F})q_{\tau}(\mathbf{F})(d\mathbf{F}),
\end{equation}
with the constant of proportionality defined by
$$
  \frac{1}{\mathcal{B}_{m}^{\beta}[\nu \beta /2, \kappa;n\beta/2, \tau]},
$$
where $\re(\nu\beta/2)> (m-1)\beta/2-k_{m}$ and $\re(n\beta/2)> (m-1)\beta/2-t_{m}$. It is said that
$\mathbf{F}$ has a \emph{matricvariate c-beta-Riesz type II distribution}.
\end{theorem}
\begin{proof}
The proof follows from (\ref{DTRI}) by applying (\ref{w}) and then (\ref{vol}).
\end{proof}

Analogously, under the conditions of Theorem \ref{teo32}, we have the following result.

\begin{theorem}\label{teo21}
The density of $\mathbf{F}$ is
\begin{equation}\label{FII2}
    \propto |\mathbf{F}|^{(n-m+1)\beta/2-1}
    |\mathbf{I}_{m}+\mathbf{F}|^{-(n+\nu)\beta/2} q_{-\kappa-\tau}\left[(\mathbf{I}_{m}+\mathbf{F})^{-1}
    \right]q_{\tau}\left[\mathbf{F^{-1}}\right](d\mathbf{F}),
\end{equation}
where the constant of proportionality is
$$
  \frac{1}{\mathcal{B}_{m}^{\beta}[\nu \beta /2, -\kappa;n\beta/2, -\tau]},
$$
where $\re(\nu\beta/2)> (m-1)\beta/2+k_{1}$ and $\re(n\beta/2)> (m-1)\beta/2+t_{1}$. It is said that
$\mathbf{F}$ has a \emph{matricvariate k-beta-Riesz type II distribution}.
\end{theorem}

Observe that, results in theorems \ref{teo2} and \ref{teo21} were obtained by \citet{dg:15b} via an
alternative way.

To conclude this section, suppose that $\u(\mathbf{\Theta}) \in \mathfrak{T}_{U}^{\beta}(m)$ is constant
matrix such that $\mathbf{\Theta} = \u(\mathbf{\Theta})^{*}\u(\mathbf{\Theta}) \in
\mathfrak{P}_{m}^{\beta}$. Also, define $\mathbf{Z} = \u(\mathbf{\Theta})^{*} \mathbf{F}
\u(\mathbf{\Theta})$, therefore:
\begin{corollary}
\begin{enumerate}
  \item If $\mathbf{F}$ has a \emph{c-beta-Riesz type II distribution}, the density of $\mathbf{Z}$ is
  \begin{equation}
    \propto |\mathbf{Z}|^{(n-m+1)\beta/2-1}
    |\mathbf{\Theta}+\mathbf{Z}|^{-(n+\nu)\beta/2} q_{-\kappa-\tau}(\mathbf{\Theta}+
    \mathbf{Z})q_{\tau}(\mathbf{Z})(d\mathbf{Z}),
  \end{equation}
  with constant of proportionality
  $$
    \frac{|\mathbf{\Theta}|^{\nu\beta/2}q_{\kappa}(\mathbf{\Theta})}{\mathcal{B}_{m}^{\beta}[\nu\beta/2,
    \kappa;n\beta/2, \tau]},
  $$
  where $\re([\nu\beta/2)> (m-1)\beta/2-k_{m}$ and $\re(m\beta/2)> (m-1)\beta/2-t_{m}$.
  $\mathbf{Z}$ is said to have a \emph{nonstandardised matricvariate c-beta-Riesz type II distribution}.
  \item If $\mathbf{F}$ has a \emph{k-beta-Riesz type II distribution}, the density of $\mathbf{Z}$ is
  \begin{equation}
    \propto |\mathbf{Z}|^{(n-m+1)\beta/2-1}
    |\mathbf{\Theta}+\mathbf{Z}|^{-(n+\nu)\beta/2}
    q_{-\kappa-\tau}\left[(\mathbf{\Theta}+\mathbf{Z})^{-1}\right]
    q_{\tau}\left[\mathbf{Z}^{-1}\right](d\mathbf{Z}),
  \end{equation}
  with constant of proportionality
  $$
    \frac{|\mathbf{\Theta}|^{\nu\beta/2}q_{\kappa}\left(\mathbf{\Theta}^{-1}\right)}
    {\mathcal{B}_{m}^{\beta}[\nu\beta/2, -\kappa;n\beta/2, -\tau]},
  $$
  where $\re([\nu\beta/2)> (m-1)\beta/2+k_{1}$ and $\re(m\beta/2)> (m-1)\beta/2+t_{1}$.
  $\mathbf{Z}$ is said to have a \emph{non standardised matricvariate k-beta-Riesz type II distribution}.
\end{enumerate}
\end{corollary}
\begin{proof}
Proof follows from (\ref{FII1}) and (\ref{FII2}), respectively, by applying (\ref{hlt}).
\end{proof}

If $k =0$ and $\tau = (0, \dots,0)$ are set in the theorems and corollaries of this section, then the
results by \citet{dggj:12} are obtained as particular cases. Finally, under similar substitutions to
those specified in (\ref{ss}), we can obtain analogous results for the random matrices $\mathbf{F}_{1} =
\mathbf{T}_{1}\mathbf{T}_{1}^{*}$ and $\mathbf{Z}_{1} = \u(\mathbf{\Theta}) \mathbf{F}_{1}
\u(\mathbf{\Theta})^{*}$, where $\u(\mathbf{\Theta})^{*} \in \mathfrak{T}_{U}^{\beta}(m)$ is such that
$\mathbf{\Theta}=\u(\mathbf{\Theta})\u(\mathbf{\Theta})^{*}$.

\section*{Conclusions}

Note that if $\tau = (p,\dots,p)$ is defined  in section \ref{sec3}, then  the corresponding results for
the matricvariate Kotz type distribution are obtained as a particular case, see \citet{fl:99}.

Although during the 90's and 2000's were obtained important results in theory of random matrices
distributions, the  past 30 years have reached a substantial development. Essentially, these advances
have been archived through two approaches based on the \emph{theory of Jordan algebras} and the \emph{
theory of real normed division algebras}. A basic source of the mathematical tools of theory of random
matrices distributions under Jordan algebras can be found in \citet{fk:94}; and specifically, some works
in the context of theory of random matrices distributions based on Jordan algebras are provided in
\citet{m:94}, \citet{cl:96}, \citet{hl:01}, \citet{hlz:05}, \citet{hlz:08} and \citet{k:14} and the
references therein. Parallel results on theory of random matrices distributions based on real normed
division algebras have been also developed in random matrix theory and statistics, see \citet{gr:87},
\citet {f:05}, \citet{dggj:11}, \citet{dggj:13}, among others. In  addition, from mathematical point of
view, several basic properties of the matrix multivariate Riesz distribution under \emph{the structure
theory of normal $j$-algebras}  and under \emph{theory of Vinberg algebras} in place of Jordan algebras
have been studied, see \citet{i:00} and \citet{bh:09}, respectively.

From a applied point of view, the relevance of \emph{the octonions} remains unclear. An excellent review
of the history, construction and many other properties of octonions is given in \citet{b:02}, where it is
stated that:
\begin{center}
\begin{minipage}[t]{4in}
\begin{sl}
``Their relevance to geometry was quite obscure until 1925, when \'Elie Cartan described `triality' --
the symmetry between vector and spinors in 8-dimensional Euclidian space. Their potential relevance to
physics was noticed in a 1934 paper by Jordan, von Neumann and Wigner on the foundations of quantum
mechanics...Work along these lines continued quite slowly until the 1980s, when it was realised that the
octionions explain some curious features of string theory... \textbf{However, there is still no
\emph{proof} that the octonions are useful for understanding the real world}. We can only hope that
eventually this question will be settled one way or another."
\end{sl}
\end{minipage}
\end{center}

For the sake of completeness, in the present article the case of octonions is considered, but the
veracity of the results obtained for this case can only be conjectured. Nonetheless, \citet[Section
1.4.5, pp. 22-24]{f:05} it is proved that the bi-dimensional density function of the eigenvalue, for a
Gaussian ensemble of a $2 \times 2$ octonionic matrix, is obtained from the general joint density
function of the eigenvalues for the Gaussian ensemble, assuming $m = 2$ and $\beta = 8$, see Section
\ref{sec2}. Moreover, as is established in \citet{fk:94} and \citet{S:97} the result obtained in this
article are valid for the \emph{algebra of Albert}, that is when hermitian matrices ($\mathbf{S}$) or
hermitian product of matrices ($\mathbf{X}^{*}\mathbf{X}$) are $3 \times 3$ octonionic matrices.

\section*{Acknowledgements}
%The authors wish to thank the Editor and the anonymous reviewers for their constructive
%comments on the preliminary version of this paper.
This research work was partially supported by IDI-Spain, Grants No. MTM2011-28962. This paper was written
during J. A. D\'{\i}az-Garc\'{\i}a's stay as a visiting professor at the Department of Statistics and O.
R. of the University of Granada, Spain.

\end{document}